\documentclass[12pt]{amsart}
\usepackage[margin=1.1in]{geometry}
\usepackage{setspace}
\usepackage{amsmath,amssymb,amsthm,amscd,mathrsfs,mathtools}
\usepackage{tikz-cd}
\usepackage[T1]{fontenc}
\usepackage[utf8]{inputenc}
\usepackage{lmodern}
\usepackage{enumitem}
\usepackage{xcolor}
\usepackage[colorlinks=true,linkcolor=blue,citecolor=blue,urlcolor=blue]{hyperref}
\usepackage{microtype}
\usepackage[nameinlink,noabbrev]{cleveref}

\newtheorem{theorem}{Theorem}[section]
\newtheorem{proposition}[theorem]{Proposition}
\newtheorem{lemma}[theorem]{Lemma}
\newtheorem{corollary}[theorem]{Corollary}

\newtheorem{definition}[theorem]{Definition}

\numberwithin{equation}{section}
\newcommand{\cO}{\mathcal{O}}

\newcommand{\etale}{\'etale}

\DeclareMathOperator{\At}{At}
\DeclareMathOperator{\End}{End}
\DeclareMathOperator{\Ext}{Ext}

\DeclareMathOperator{\longrightarrowt}{Tot}

\numberwithin{equation}{section}
\allowdisplaybreaks

\title[Logarithmic Atiyah classes and spectral data]{Spectral and Logarithmic Atiyah Classes for Higgs Bundles}
\author[P. Kumar]{Pradip Kumar}
\address{Department of Mathematics, Shiv Nadar University, Dadri 201314, Uttar Pradesh, India}
\email{pradip.kumar@snu.edu.in}
\author[S. R. R. Mohanty]{Sai Rasmi Ranjan Mohanty}
\address{Department of Mathematics and Statistics, Indian Institute of Technology Kanpur, Uttar Pradesh 208016, India}
\email{sairrm@iitk.ac.in}
\author[S. Rani]{Savita Rani}
\address{International Center for Theoretical Sciences, Tata Institute of Fundamental Research, Bengaluru-560089, India}
\email{mansavi.14@gmail.com}
\author[R. K. Singh]{Rahul Kumar Singh}
\address{Department of Mathematics, Indian Institute of Technology Patna, Bihta 801106, Bihar, India}
\email{rahulks@iitp.ac.in}

\begin{document}
\begin{abstract}
For a regular semisimple Higgs bundle with a smooth spectral curve, we prove that, over the \etale\ locus, the Atiyah class of the underlying bundle is induced by the Atiyah class of the spectral line bundle and takes values in the centralizer of the Higgs field. Further, when the discriminant is reduced, we construct a logarithmic refinement across the branch divisor: the Atiyah class extends as a class with logarithmic poles and values in a natural regularized centralizer sheaf.
\end{abstract}

\maketitle
\section{Introduction}\label{sec:intro}

The Atiyah class of a holomorphic vector bundle is a basic cohomological invariant of its holomorphic structure. Introduced by Atiyah through the first jet sequence \cite{Atiyah57}, it measures the obstruction to the existence of a holomorphic connection. Since then, Atiyah classes have played an important role in deformation theory, Riemann--Roch formalisms, and several related constructions in complex and algebraic geometry; see for instance \cite{Markarian09, HuybrechtsThomas10, Kapranov99}. For a vector bundle \(F\) on a smooth complex projective curve \(C\), the class
\[
\At(F)\,\,\in\,\, \Ext_C^1(F,\,\Omega_C^1\otimes F)\,\cong\, H^1(C,\,\Omega_C^1\otimes \End(F))
\]
is represented by the extension class of the first jet sequence
\[
0\longrightarrow \Omega_C^1\otimes F\longrightarrow J_C^1(F)\longrightarrow F\longrightarrow 0.
\]
A natural question is how this class behaves when a higher-rank bundle is obtained from a simpler covering or spectral data.

Higgs bundles provide a natural setting for this question. Since the foundational work of Hitchin and Simpson \cite{Hitchin87, Simpson92, Simpson94}, Higgs bundles on a smooth projective curve have been studied via the Hitchin fibration, spectral curves, and non-abelian Hodge theory. In the regular semisimple case, the Beauville--Narasimhan--Ramanan correspondence identifies a Higgs bundle with line bundle data on its spectral curve \cite{BNR89}; see also Donagi's broader spectral-cover viewpoint \cite{Donagi95}. Concretely, if \((E,\,\,\theta)\) is a Higgs bundle on \(X\) with smooth spectral curve
\[
\pi\,\colon\, Y\hookrightarrow \longrightarrowt(\Omega_X^1)\longrightarrow X,
\]
then
\[
E\,\cong\, \pi_*L
\]
for a line bundle \(L\) on \(Y\), and the Higgs field is recovered from the tautological section on the total space of \(\Omega_X^1\).

This spectral description turns many questions about Higgs bundles into questions about the line bundle \(L\) on \(Y\). It is therefore natural to ask whether the Atiyah class of the underlying bundle \(E\) can also be read directly from the spectral line bundle. The basic difficulty is that the relevant classes live in different spaces and with different coefficient sheaves:
\[
\At(E)\,\in\, H^1(X,\,\,\Omega_X^1\otimes \End(E)),
\qquad
\At(L)\,\in\, H^1(Y,\,\,\Omega_Y^1).
\]
Thus, before comparing them, one needs a precise relation between the first jet sequence of \(E=\pi_*L\) on \(X\) and the pushed-forward first jet sequence of \(L\) on \(Y\).

Our first main result provides exactly this bridge. Let \(\pi\colon Y\longrightarrow X\) be a finite morphism of smooth complex projective curves, let \(L\) be a line bundle on \(Y\), and set \(E:=\pi_*L\). In Proposition \ref{prop:jet-comparison}, we construct a canonical morphism
\[
\Psi_{\pi,\,L}\colon J_X^1(E)\longrightarrow \pi_*J_Y^1(L)
\]
compatible with the two jet sequences, and we prove that it fits into a short exact sequence
\[
0\longrightarrow J_X^1(E)\xrightarrow{\ \Psi_{\pi,\,L}\ }
\pi_*J_Y^1(L)\longrightarrow \pi_*(\Omega_{Y/X}^1\otimes L)\longrightarrow 0.
\]
Thus, the discrepancy between first jets on \(X\) and pushed-forward first jets on \(Y\) is measured exactly by the ramification sheaf \(\pi_*(\Omega_{Y/X}^1\otimes L)\).

From this comparison, one obtains a canonical extension class determined only by the spectral data:
\[
\beta_{\mathrm{spec}}(L,\,\pi)\,\in\,
\Ext_X^1\bigl(E,\,\,\pi_*(\Omega_Y^1\otimes L)\bigr),
\]
represented by the pushed-forward jet sequence of \(L\). We call this the \emph{spectral Atiyah class}. Corollary \ref{cor:global-bridge} shows that the ordinary Atiyah class of \(E\) maps to \(\beta_{\mathrm{spec}}(L,\,\pi)\) through the natural coefficient morphism
\[
\iota\,\colon\, \Omega_X^1\otimes E\longrightarrow \pi_*(\Omega_Y^1\otimes L).
\]
In other words, the spectral Atiyah class is precisely the image of the ordinary Atiyah class after changing coefficients from differentials on \(X\) to pushed-forward differentials on \(Y\).

We then specialize to the case of regular semisimple Higgs bundles. Let \(\Delta\subset X\) be the discriminant divisor, set
\[
X^\circ\,:=\,X\setminus \Delta,
\qquad
Y^\circ\,:=\,\pi^{-1}(X^\circ),
\]
and write
\[
j\,\colon\, X^\circ\hookrightarrow X,
\qquad
\pi^\circ\colon Y^\circ\longrightarrow X^\circ.
\]
Then \(\pi^\circ\) is finite \etale. On this open set, the restricted spectral class admits a scalar form
\[
\beta^\circ_{\mathrm{spec}}(L,\,\pi)\,\in\,
H^1\bigl(X^\circ,\,\Omega_{X^\circ}^1\otimes \pi^\circ_*\cO_{Y^\circ}\bigr),
\]
obtained from \(\At(L^\circ)\) using the identifications
\[
\Omega_{Y^\circ}^1\,\cong\, \pi^{\circ *}\Omega_{X^\circ}^1,
\qquad
\End(L^\circ)\,\cong\, \cO_{Y^\circ}.
\]
Thus \(\beta^\circ_{\mathrm{spec}}(L,\,\pi)\) is the scalar form of the restriction of the global spectral Atiyah class to the \etale\ locus.

\medskip 

Our second main result, Theorem \ref{thm:etale-invariant}, shows that on \(X^\circ\) the Atiyah class of \(E\) is induced from this scalar spectral class through the spectral action:
\[
(\operatorname{id}\otimes m^\circ)_*
\beta^\circ_{\mathrm{spec}}(L,\,\pi)
=
j^*\At(E)
\]
in
\[
H^1\bigl(X^\circ,\Omega_{X^\circ}^1\otimes \End(E|_{X^\circ})\bigr).
\]
Equivalently, away from the discriminant, the Atiyah class of the underlying bundle is completely governed by the spectral line bundle and takes values, through the spectral action, in the centralizer of the Higgs field.

\medskip

The previous statement naturally suggests a further question: whether the spectral description extends across the discriminant in a controlled way. Our third main result answers this under the additional assumption that the discriminant divisor is reduced. In Theorem \ref{thm:logarithmic-centralizer}, we construct a canonical logarithmic spectral Atiyah class
\[
\At^{\log}_{\theta}(E)\,\in\,
H^1\bigl(X,\,\,\Omega_X^1(\log \Delta)\otimes \mathcal{Z}_\theta\bigr),
\]
where \(\mathcal{Z}_\theta\subset \End(E)\) is the regularized centralizer sheaf obtained by saturating the image of the spectral algebra \(\pi_*\cO_Y\) inside \(\End(E)\). This sheaf agrees over \(X^\circ\) with the usual centralizer of the Higgs field and gives a natural extension of it across the discriminant.

The logarithmic class is compatible with the ordinary Atiyah class in the following global sense:
\[
(k\otimes \operatorname{id}_{\End(E)})_*\At(E)
=
(\operatorname{id}\otimes q)_*\At^{\log}_{\theta}(E)
\]
in
\[
H^1\bigl(X,\,\Omega_X^1(\log\Delta)\otimes \End(E)\bigr),
\]
where
\[
k:\Omega_X^1\hookrightarrow \Omega_X^1(\log\Delta),
\qquad
q:\mathcal Z_\theta\hookrightarrow \End(E)
\]
are the natural inclusions. After restriction to \(X^\circ\), this logarithmic class recovers the \etale\ spectral description of the ordinary Atiyah class:
\[
(\operatorname{id}\otimes q^\circ)_*j^*\At^{\log}_{\theta}(E)
=
(\operatorname{id}\otimes m^\circ)_*
\beta^\circ_{\mathrm{spec}}(L,\,\pi)
=
j^*\At(E).
\]
Thus, the logarithmic class extends the abelianized spectral description across the branch divisor, with only logarithmic singularities and with coefficients in the regularized centralizer.

\medskip 

The picture that emerges is the following. On the \etale\ locus, the Atiyah class of the underlying bundle is completely abelianized by the spectral correspondence. Globally, the failure of this abelianized description at the level of ordinary first jets is measured by the ramification sheaf \(\pi_*(\Omega_{Y/X}^1\otimes L)\). When the discriminant is reduced, this failure admits a canonical logarithmic reformulation: the ordinary Atiyah class, after allowing logarithmic poles along the discriminant, is represented by a class with values in the regularized centralizer sheaf.

\medskip 

In Section \ref{sec:jets} we prove the jet comparison theorem, define the spectral Atiyah class, and relate it to the ordinary Atiyah class of \(\pi_*L\). In Section \ref{sec:spectral} we specialize to regular semisimple Higgs bundles with smooth spectral curve, describe the \etale\ locus, and show that there the Atiyah class is induced from the scalar Atiyah class of the spectral line bundle through the spectral action. In Section \ref{sec:logarithmic} we study the extension across the discriminant: assuming the discriminant is reduced, we construct the logarithmic class \(\At^{\log}_\theta(E)\) with values in the regularized centralizer sheaf and prove its global compatibility with the ordinary Atiyah class after passing to logarithmic coefficients.
\section{Atiyah classes, jets, and finite pushforward}\label{sec:jets}
For a vector bundle \(F\) on a smooth complex projective curve \(C\), let
\begin{equation}\label{eq:jet-sequence}
0\longrightarrow \Omega_C^1\otimes F \longrightarrow J_C^1(F)\longrightarrow F\longrightarrow 0
\end{equation}
be its first jet sequence. The associated extension class is the Atiyah class
\[
\At(F)\,\in\, \Ext_C^1(F,\,\Omega_C^1\otimes F)\,\cong\, H^1\!\left(C,\,\Omega_C^1\otimes \End(F)\right).
\]
If \(F\) is a line bundle, this class lies in \(H^1(C,\,\Omega_C^1)\).

\smallskip 

Now fix a finite morphism
\(
\pi\,\colon\, Y\longrightarrow X
\)
of smooth complex projective curves, a line bundle \(L\) on \(Y\), and set
\[
E\,:=\,\pi_*L.
\]
Since \(\pi\) is finite and \(X,\,\,Y\) are smooth curves, \(E\) is a vector bundle on \(X\). The natural map of differentials
\[
\pi^*\Omega_X^1\longrightarrow \Omega_Y^1
\]
is injective, so after tensoring with \(L\) and pushing forward, the projection formula gives a canonical morphism
\[
\iota\,\,\colon\,\, \Omega_X^1\otimes E
\;\,\cong\,\;
\pi_*(\pi^*\Omega_X^1\otimes L)
\hookrightarrow
\pi_*(\Omega_Y^1\otimes L).
\]

\begin{proposition}\label{prop:jet-comparison}
With notation as above, there is a canonical morphism
\(
\Psi_{\pi,\,L}\,\colon\, J_X^1(E)\longrightarrow \pi_*J_Y^1(L)
\)
such that the diagram
\begin{equation}\label{eq:comparison-diagram}
\begin{CD}
0 @>>> \Omega_X^1\otimes E @>>> J_X^1(E) @>>> E @>>> 0\\
@. @V{\iota}VV @V{\Psi_{\pi,\,L}}VV @| @.\\
0 @>>> \pi_*(\Omega_Y^1\otimes L) @>>> \pi_*J_Y^1(L) @>>> E @>>> 0
\end{CD}
\end{equation}
commutes, where both rows are exact. Moreover, there is  a short exact sequence
\begin{equation}\label{eq:jet-comparison-exact}
0\longrightarrow J_X^1(E)\xrightarrow{\ \Psi_{\pi,\,L}\ }
\pi_*J_Y^1(L)\longrightarrow \pi_*(\Omega_{Y/X}^1\otimes L)\longrightarrow 0.
\end{equation}
\end{proposition}
\begin{proof}
Let
\(
\varepsilon\,\colon\, \pi^*E\,=\,\pi^*\pi_*L\longrightarrow L
\)
be a canonical map. We first construct the morphism
\[
\Psi_{\pi,\,L}\,\colon\, J_X^1(E)\longrightarrow \pi_*J_Y^1(L).
\]

Since \(X\) and \(Y\) are smooth curves and \(\pi\,\colon\, Y\longrightarrow X\) is finite, the morphism
\(\pi\) is flat. Hence pulling back a short exact sequence of locally free sheaves on \(X\)
remains exact on \(Y\). Applying \(\pi^*\) to the first jet sequence of \(E\),
\[
0\longrightarrow \Omega_X^1\otimes E
\longrightarrow J_X^1(E)
\longrightarrow E
\longrightarrow 0,
\]
we obtain an exact sequence on \(Y\):
\[
0\longrightarrow \pi^*\Omega_X^1\otimes \pi^*E
\longrightarrow \pi^*J_X^1(E)
\longrightarrow \pi^*E
\longrightarrow 0.
\]
Next, the differential of \(\pi\) gives a natural morphism of sheaves
\[
d\pi\,\colon\, \pi^*\Omega_X^1\longrightarrow \Omega_Y^1.
\]
Tensoring with \(\pi^*E\) yields
\[
d\pi\otimes \mathrm{id}_{\pi^*E}\,\colon\,
\pi^*\Omega_X^1\otimes \pi^*E\longrightarrow \Omega_Y^1\otimes \pi^*E.
\]
By the functoriality of the jet bundle construction, this induces a morphism of exact
sequences
\[
\begin{CD}
0 @>>> \pi^*\Omega_X^1\otimes \pi^*E @>>> \pi^*J_X^1(E) @>>> \pi^*E @>>> 0\\
@. @V{d\pi\otimes \mathrm{id}}VV @VVV @| @.\\
0 @>>> \Omega_Y^1\otimes \pi^*E @>>> J_Y^1(\pi^*E) @>>> \pi^*E @>>> 0.
\end{CD}
\]

\medskip

Now apply functoriality once more to the morphism
\(
\varepsilon\,\colon\, \pi^*E\longrightarrow L.
\)
It induces a morphism of jet sequences
\[
\begin{CD}
0 @>>> \Omega_Y^1\otimes \pi^*E @>>> J_Y^1(\pi^*E) @>>> \pi^*E @>>> 0\\
@. @V{\mathrm{id}\otimes \varepsilon}VV @VVV @V{\varepsilon}VV @.\\
0 @>>> \Omega_Y^1\otimes L @>>> J_Y^1(L) @>>> L @>>> 0.
\end{CD}
\]

Composing the two vertical morphisms in the middle column, we obtain a canonical morphism
on \(Y\),
\[
\Phi\,\colon\, \pi^*J_X^1(E)\longrightarrow J_Y^1(L).
\]
By adjunction, \(\Phi\) corresponds uniquely to a morphism on \(X\),
\[
\Psi_{\pi,\,L}\,\colon\, J_X^1(E)\longrightarrow \pi_*J_Y^1(L).
\]
This is the desired map.

\medskip

We now identify the induced map on the left-hand terms. The morphism on the left in the
composite diagram is
\[
\pi^*\Omega_X^1\otimes \pi^*E
\xrightarrow{\,d\pi\otimes \mathrm{id}\,}
\Omega_Y^1\otimes \pi^*E
\xrightarrow{\,\mathrm{id}\otimes \varepsilon\,}
\Omega_Y^1\otimes L.
\]
Using the canonical associativity isomorphism
\(
\pi^*\Omega_X^1\otimes \pi^*E
\,\cong\,
\pi^*\Omega_X^1\otimes \pi^*\pi_*L
\)
and the adjunction map \(\varepsilon\,\colon\, \pi^*\pi_*L\longrightarrow L\), this factors as
\[
\pi^*\Omega_X^1\otimes \pi^*E
\longrightarrow
\pi^*\Omega_X^1\otimes L
\xrightarrow{\,d\pi\otimes \mathrm{id}_L\,}
\Omega_Y^1\otimes L.
\]
Pushing forward to \(X\), and using the projection formula,
\[
\pi_*(\pi^*\Omega_X^1\otimes L)
\,\cong\,
\Omega_X^1\otimes \pi_*L
\,=\,
\Omega_X^1\otimes E,
\]
we obtain a canonical morphism
\[
\iota\,\colon\, \Omega_X^1\otimes E\longrightarrow \pi_*(\Omega_Y^1\otimes L).
\]
By construction, this is exactly the left vertical arrow induced by \(\Psi_{\pi,\,L}\).
Therefore \(\Psi_{\pi,\,L}\) fits into the commutative diagram
\[
\begin{CD}
0 @>>> \Omega_X^1\otimes E @>>> J_X^1(E) @>>> E @>>> 0\\
@. @V{\iota}VV @V{\Psi_{\pi,\,L}}VV @| @.\\
0 @>>> \pi_*(\Omega_Y^1\otimes L) @>>> \pi_*J_Y^1(L) @>>> E @>>> 0,
\end{CD}
\]
where the top row is the jet sequence on \(X\), and the bottom row is obtained by pushing
forward the jet sequence
\[
0\longrightarrow \Omega_Y^1\otimes L
\longrightarrow J_Y^1(L)
\longrightarrow L
\longrightarrow 0
\]
along the finite morphism \(\pi\). Since \(\pi\) is finite, \(\pi_*\) is exact, so the
bottom row is also exact.

\medskip

It remains to prove that \(\Psi_{\pi,\,L}\) is injective and to identify its cokernel. For a finite morphism of smooth curves, the standard exact sequence of K\"ahler
differentials is
\[
0\longrightarrow \pi^*\Omega_X^1
\xrightarrow{\,d\pi\,}
\Omega_Y^1
\longrightarrow \Omega_{Y/X}^1
\longrightarrow 0.
\]
Because \(L\) is a line bundle, tensoring with \(L\) preserves exactness, and we get
\[
0\longrightarrow \pi^*\Omega_X^1\otimes L
\longrightarrow \Omega_Y^1\otimes L
\longrightarrow \Omega_{Y/X}^1\otimes L
\longrightarrow 0.
\]
Applying \(\pi_*\), which is exact for finite morphisms, we obtain
\[
0\longrightarrow \pi_*(\pi^*\Omega_X^1\otimes L)
\longrightarrow \pi_*(\Omega_Y^1\otimes L)
\longrightarrow \pi_*(\Omega_{Y/X}^1\otimes L)
\longrightarrow 0.
\]
Using the projection formula once again,
\[
\pi_*(\pi^*\Omega_X^1\otimes L)\,\cong\, \Omega_X^1\otimes \pi_*L\,=\,\Omega_X^1\otimes E,
\]
this becomes
\[
0\longrightarrow \Omega_X^1\otimes E
\xrightarrow{\ \iota\ }
\pi_*(\Omega_Y^1\otimes L)
\longrightarrow \pi_*(\Omega_{Y/X}^1\otimes L)
\longrightarrow 0.
\]

Now apply the snake lemma to the commutative diagram
\[
\begin{CD}
0 @>>> \Omega_X^1\otimes E @>>> J_X^1(E) @>>> E @>>> 0\\
@. @V{\iota}VV @V{\Psi_{\pi,\,L}}VV @| @.\\
0 @>>> \pi_*(\Omega_Y^1\otimes L) @>>> \pi_*J_Y^1(L) @>>> E @>>> 0.
\end{CD}
\]
Since the rightmost vertical map is the identity on \(E\), the snake lemma gives
\[
0\longrightarrow \ker(\Psi_{\pi,\,L})
\longrightarrow \ker(\mathrm{id}_E)\,=\,0,
\]
hence
\(
\ker(\Psi_{\pi,\,L})\,=\,0.
\)
So \(\Psi_{\pi,\,L}\) is injective. The same snake lemma gives an exact sequence
\[
0\longrightarrow \operatorname{coker}(\iota)
\longrightarrow \operatorname{coker}(\Psi_{\pi,\,L})
\longrightarrow \operatorname{coker}(\mathrm{id}_E)\,=\,0,
\]
so
\(
\operatorname{coker}(\Psi_{\pi,\,L})\,\cong\, \operatorname{coker}(\iota).
\)
From the exact sequence above for \(\iota\), we identify
\(
\operatorname{coker}(\iota)\,\cong\, \pi_*(\Omega_{Y/X}^1\otimes L).
\)
Therefore
\[
\operatorname{coker}(\Psi_{\pi,\,L})
\,\cong\,
\pi_*(\Omega_{Y/X}^1\otimes L).
\]

We conclude that there is a short exact sequence
\[
0\longrightarrow J_X^1(E)
\xrightarrow{\ \Psi_{\pi,\,L}\ }
\pi_*J_Y^1(L)
\longrightarrow \pi_*(\Omega_{Y/X}^1\otimes L)
\longrightarrow 0,
\]
as claimed.
\end{proof}
We define the \emph{spectral Atiyah class} of the spectral data \((L,\,\pi)\) by
\[
\beta_{\mathrm{spec}}(L,\,\pi)
\;:=\;
\pi_*\At(L)
\;\in\;
\Ext_X^1\bigl(E,\,\,\pi_*(\Omega_Y^1\otimes L)\bigr),
\qquad E=\pi_*L.
\]
Equivalently, \(\beta_{\mathrm{spec}}(L,\,\pi)\) is represented by the pushed-forward first jet sequence
\begin{equation}\label{eq:spectral-atiyah-class}
0 \longrightarrow \pi_*(\Omega_Y^1\otimes L)
\longrightarrow \pi_*J_Y^1(L)
\longrightarrow E
\longrightarrow 0.
\end{equation}

If \(u\colon A\longrightarrow   B\) is a morphism of sheaves on \(X\), let
\[
u_* \;:=\; \Ext_X^1(E,\,\,u)\colon \Ext_X^1(E,A)\longrightarrow \Ext_X^1(E,B)
\]
denote the induced map. In particular, for
\[
\iota\colon \Omega_X^1\otimes E \longrightarrow \pi_*(\Omega_Y^1\otimes L),
\]
the class \(\iota_*\At(E)\) is the image of \(\At(E)\) in
\[
\Ext_X^1\bigl(E,\,\,\pi_*(\Omega_Y^1\otimes L)\bigr).
\]

\begin{corollary}\label{cor:global-bridge}
With notation as above,
\begin{equation}\label{eq:global-bridge}
\iota_*\At(E)=\beta_{\mathrm{spec}}(L,\,\pi)
\qquad\text{in}\qquad
\Ext_X^1\bigl(E,\,\,\pi_*(\Omega_Y^1\otimes L)\bigr).
\end{equation}
\end{corollary}

\begin{proof}
By Proposition~\ref{prop:jet-comparison}, the first jet sequence of \(E\) and the sequence
\eqref{eq:spectral-atiyah-class} fit into a commutative diagram
\[
\begin{tikzcd}
0 \arrow[r] &
\Omega_X^1\otimes E \arrow[r] \arrow[d,"\iota"] &
J_X^1(E) \arrow[r] \arrow[d,"\Psi_{\pi,\,L}"] &
E \arrow[r] \arrow[d,equal] &
0 \\
0 \arrow[r] &
\pi_*(\Omega_Y^1\otimes L) \arrow[r] &
\pi_*J_Y^1(L) \arrow[r] &
E \arrow[r] &
0.
\end{tikzcd}
\]
Since the right vertical arrow is the identity on \(E\), the lower extension is the pushout
of the upper one along \(\iota\). Hence, its extension class is \(\iota_*\At(E)\). By
definition, the lower row represents \(\beta_{\mathrm{spec}}(L,\,\pi)\). This proves
\eqref{eq:global-bridge}.
\end{proof}

Thus the spectral Atiyah class is the image of the ordinary Atiyah class under the natural coefficient morphism \(\iota\). In the next section, we restrict this relation to the \etale\ locus and rewrite the restricted class in scalar form.

\section{Regular semisimple Higgs bundles and the \etale\ locus}\label{sec:spectral}
Let $X$ be a smooth complex projective curve, and let $(E,\,\,\theta)$ be a regular semisimple Higgs bundle on $X$ with smooth spectral curve
\begin{equation}\label{eq:spectral-cover}
\pi\,\colon\, Y\hookrightarrow \longrightarrowt(\Omega^{1}_{X})\longrightarrow X
\end{equation}
and spectral line bundle $L$ such that $E\,\cong\, \pi_{*}L$.

Let $\Delta\subset X$ be the discriminant divisor, set $X^{\circ}\,:=\,X\setminus \Delta$, and write $j\,\colon\, X^{\circ}\hookrightarrow X$ for the inclusion. Let $Y^{\circ}\,:=\,\pi^{-1}(X^{\circ})$ and denote by $\pi^{\circ}\,\colon\, Y^{\circ}\longrightarrow X^{\circ}$ the induced map. Then $\pi^{\circ}$ is finite \etale.

Multiplication on sections gives a natural morphism of sheaves of $\cO_{X}$-algebras
\begin{equation}\label{eq:multiplication-map}
m\,\colon\, \pi_{*}\cO_{Y}\longrightarrow \End(E).
\end{equation}
Its restriction to $X^{\circ}$ is the usual spectral action.

\begin{lemma}\label{lem:centralizer}
Over $X^{\circ}$ the multiplication map induces an isomorphism of sheaves of $\cO_{X^{\circ}}$-algebras
\[
m^{\circ}\,\colon\, \pi^{\circ}_{*}\cO_{Y^{\circ}}\xrightarrow{\sim} Z_{\End(E|_{X^{\circ}})}(\theta|_{X^{\circ}}),
\]
where $Z_{\End(E|_{X^{\circ}})}(\theta|_{X^{\circ}})$ denotes the centralizer sheaf of the Higgs field.
\end{lemma}
\begin{proof}
On a simply connected open subset $U\subset X^{\circ}$, the cover splits as a disjoint union of sheets
\[
\pi^{-1}(U)\,=\,U_{1}\sqcup \cdots \sqcup U_{n},
\qquad
\pi_{i}\,:=\,\pi|_{U_{i}}\,\colon\, U_{i}\xrightarrow{\sim} U.
\]
Hence
\[
E|_{U}\,\cong\, \bigoplus_{i\,=\,1}^{n}(\pi_{i})_{*}(L|_{U_{i}}),
\]
and $\theta|_{U}$ is diagonal with pairwise distinct eigenvalues. The endomorphisms commuting with $\theta|_{U}$ are therefore exactly the diagonal endomorphisms, which identify with functions on the disjoint union $U_{1}\sqcup\cdots\sqcup U_{n}$. These local identifications agree on overlaps.
\end{proof}

\medskip

Let $L^{\circ}:=L|_{Y^{\circ}}$. Restricting the global spectral Atiyah class gives
\[
j^*\beta_{\mathrm{spec}}(L,\,\pi)\,\in\,
\Ext^1_{X^{\circ}}\bigl(E|_{X^{\circ}},\,\pi^{\circ}_{*}(\Omega^{1}_{Y^{\circ}}\otimes L^{\circ})\bigr).
\]
Since $\pi^{\circ}$ is finite \etale, we have $\Omega^{1}_{Y^{\circ}}\,\cong\, \pi^{\circ *}\Omega^{1}_{X^{\circ}}$, and since $L^{\circ}$ is a line bundle, $\End(L^{\circ})\,\cong\, \cO_{Y^{\circ}}$. Therefore, the Atiyah class
\[
\At(L^{\circ})\,\in\, H^{1}(Y^{\circ},\Omega^{1}_{Y^{\circ}})
\]
corresponds, via the canonical isomorphisms
\[
H^{1}(Y^{\circ},\Omega^{1}_{Y^{\circ}})
\,\cong\,
H^{1}(Y^{\circ},\pi^{\circ *}\Omega^{1}_{X^{\circ}})
\,\cong\,
H^{1}\bigl(X^{\circ},\Omega^{1}_{X^{\circ}}\otimes \pi^{\circ}_{*}\cO_{Y^{\circ}}\bigr),
\]
to a class
\[
\beta^{\circ}_{\mathrm{spec}}(L,\,\pi)\,\in\,
H^{1}\bigl(X^{\circ},\Omega^{1}_{X^{\circ}}\otimes \pi^{\circ}_{*}\cO_{Y^{\circ}}\bigr).
\]

\begin{definition}\label{def:beta-spec-etale}
We call $\beta^{\circ}_{\mathrm{spec}}(L,\,\pi)$ the scalar form of the restricted spectral Atiyah class on the \etale\ locus.
\end{definition}

\begin{theorem}\label{thm:etale-invariant}
Let $(E,\,\,\theta)$, $Y$, and $L$ be as above. Then the explicit map of coefficient sheaves
\begin{equation}\label{eq:centralizer-inclusion}
\operatorname{id}\otimes m^{\circ}\,\colon\, \Omega^{1}_{X^{\circ}}\otimes \pi^{\circ}_{*}\cO_{Y^{\circ}}\longrightarrow \Omega^{1}_{X^{\circ}}\otimes \End(E|_{X^{\circ}})
\end{equation}
sends the scalar form of the restricted spectral Atiyah class to the restriction of the Atiyah class:
\begin{equation}\label{eq:etale-main}
(\operatorname{id}\otimes m^{\circ})_{*}\beta^{\circ}_{\mathrm{spec}}(L,\,\pi)\,=\,j^{*}\At(E)
\qquad \text{in } H^{1}\bigl(X^{\circ},\Omega^{1}_{X^{\circ}}\otimes \End(E|_{X^{\circ}})\bigr).
\end{equation}
\end{theorem}

\begin{proof}
Restrict Proposition \ref{prop:jet-comparison} to $X^{\circ}$. Since $\Omega^{1}_{Y^{\circ}/X^{\circ}}=0$, the exact sequence \eqref{eq:jet-comparison-exact} becomes an isomorphism
\[
\Psi_{\pi^{\circ},L^{\circ}}\colon J^{1}_{X^{\circ}}(E|_{X^{\circ}})\xrightarrow{\sim} \pi^{\circ}_{*}J^{1}_{Y^{\circ}}(L^{\circ}).
\]
Hence, the coefficient inclusion $\iota^{\circ}$ is an isomorphism. Restricting Corollary \ref{cor:global-bridge} to $X^{\circ}$, we obtain
\[
j^{*}\At(E)=j^{*}\beta_{\mathrm{spec}}(L,\,\pi)
\]
after identifying
\(
\Omega^{1}_{X^{\circ}}\otimes E|_{X^{\circ}}
\,\cong\,
\pi^{\circ}_{*}(\Omega^{1}_{Y^{\circ}}\otimes L^{\circ}).
\)

Now $L^{\circ}$ is a line bundle, so $\End(L^{\circ})\,\cong\, \cO_{Y^{\circ}}$ and \(\At(L^{\circ})\) is a scalar class. By Definition \ref{def:beta-spec-etale}, this scalar class corresponds to
\[
\beta^{\circ}_{\mathrm{spec}}(L,\,\pi)\,\in\,
H^{1}\bigl(X^{\circ},\Omega^{1}_{X^{\circ}}\otimes \pi^{\circ}_{*}\cO_{Y^{\circ}}\bigr).
\]
Under multiplication on \(E|_{X^{\circ}}\), the action of \(\pi^{\circ}_{*}\cO_{Y^{\circ}}\) on \(E|_{X^{\circ}}\) becomes the map \eqref{eq:centralizer-inclusion}. Therefore the image of \(\beta^{\circ}_{\mathrm{spec}}(L,\,\pi)\) under \((\operatorname{id}\otimes m^\circ)_*\) is exactly the class represented by the pushed-forward jet sequence on \(X^\circ\), namely \(j^*\At(E)\). This proves \eqref{eq:etale-main}.
\end{proof}

\section{A logarithmic refinement across the discriminant}\label{sec:logarithmic}

The results of Section~\ref{sec:spectral} show that over the \etale\ locus
\[
X^\circ:=X\setminus \Delta
\]
the Atiyah class of the underlying bundle is induced from the scalar Atiyah class of the
spectral line bundle via the spectral action; see Theorem~\ref{thm:etale-invariant}. In
this section, we explain how this spectral description extends across the discriminant
with logarithmic singularities when the discriminant divisor is reduced.

Throughout this section, \(\Delta\) denotes the scheme-theoretic discriminant divisor of
the finite map
\[
\pi:Y\longrightarrow X,
\]
equivalently, the branch divisor defined by the zeroth Fitting ideal of
\(\pi_*\Omega^1_{Y/X}\). In the spectral curve situation, this agrees with the divisor cut
out by the discriminant of the characteristic polynomial.

\begin{definition}\label{def:log-centralizer}
Assume that the discriminant divisor \(\Delta\subset X\) is reduced. We define the
\emph{regularized centralizer sheaf} of \(\theta\) to be the saturation
\[
\mathcal Z_\theta
:=
\operatorname{Sat}_{\End(E)}\bigl(m(\pi_*(\mathcal{O}_Y))\bigr)
=
\ker\!\left(
\End(E)\longrightarrow
\frac{\End(E)/m(\pi_*(\mathcal{O}_Y)}
{\text{\rm torsion}}
\right).
\]
\end{definition}

Since \(X\) is a smooth curve, every saturated subsheaf of a vector bundle is locally
free. Hence \(\mathcal Z_\theta\) is a vector bundle on \(X\). By
Lemma~\ref{lem:centralizer}, over the \etale\ locus one has
\[
\mathcal Z_\theta|_{X^\circ}
\cong
Z_{\End(E|_{X^\circ})}(\theta|_{X^\circ}).
\]
Thus \(\mathcal Z_\theta\) is a natural locally free extension across \(\Delta\) of the
centralizer sheaf of the Higgs field on \(X^\circ\).   We now construct the logarithmic coefficient morphism.

\begin{lemma}\label{lem:log-coefficient-map}
Assume that \(\Delta\) is reduced. There is a canonical \(\pi_*(\mathcal{O}_Y)\)-linear morphism
\[
\lambda_\pi:
\pi_*\Omega_Y^1
\longrightarrow
\Omega_X^1(\log\Delta)\otimes \pi_*(\mathcal{O}_Y)
\]
whose restriction to \(X^\circ\) is the usual \etale\ identification
\(
\pi^\circ_*\Omega_{Y^\circ}^1
\cong
\Omega_{X^\circ}^1\otimes \pi^\circ_*\cO_{Y^\circ}.
\)
Moreover, if
\(
d\pi_{\pi_*(\mathcal{O}_Y)}:
\Omega_X^1\otimes \pi_*(\mathcal{O}_Y)
\longrightarrow
\pi_*\Omega_Y^1
\)
is the morphism induced by \(d\pi:\pi^*\Omega_X^1\to\Omega_Y^1\), and if
\[
k:\Omega_X^1\hookrightarrow \Omega_X^1(\log\Delta)
\]
is the natural inclusion, then
\begin{equation}\label{eq:lambda-dpi-compatibility}
\lambda_\pi\circ d\pi_{\pi_*(\mathcal{O}_Y)}
=
k\otimes \operatorname{id}_{\pi_*(\mathcal{O}_Y)}.
\end{equation}
Consequently, after composing with the spectral action, there is a canonical morphism
\[
\mu_\pi:
\pi_*\Omega_Y^1
\longrightarrow
\Omega_X^1(\log\Delta)\otimes \mathcal Z_\theta .
\]
\end{lemma}

\begin{proof}
Over \(X^\circ\), the morphism is the canonical identification
\[
\Omega_{Y^\circ}^1
\cong
\pi^{\circ *}\Omega_{X^\circ}^1.
\]
We need to check that this identification extends across \(\Delta\) with at most
logarithmic poles.

Let \(x\in\Delta\). Since \(\Delta\) is reduced, the branch multiplicity at \(x\) is one.
Equivalently, over a sufficiently small analytic disc \(D\subset X\) centered at \(x\),
there is exactly one simply ramified component and all other components are unramified.
On the simply ramified component we may choose local coordinates \(z\) on \(D\) and \(t\)
on the corresponding component of \(\pi^{-1}(D)\) such that
\[
z=t^2.
\]
Thus, the local algebra of this component is
\[
\cO_D\oplus t\cO_D,
\qquad t^2=z.
\]

On this ramified component, every local section of \(\Omega_Y^1\) can be written uniquely
in the form
\[
\eta=(a_0(z)+t\,a_1(z))\,dt,
\qquad a_0,a_1\in \cO_D.
\]
Since
\[
dt=\frac{dz}{2t},
\]
we obtain
\[
\eta
=
\frac{1}{2}a_1(z)\,dz
+
\frac{1}{2}a_0(z)\frac{dz}{t}.
\]
Using \(1/t=t/z\), this becomes
\[
\eta
\longmapsto
\frac{1}{2}a_1(z)\,dz\otimes 1
+
\frac{1}{2}a_0(z)\frac{dz}{z}\otimes t
\]
as a section of
\[
\Omega_X^1(\log\Delta)\otimes \pi_*(\mathcal{O}_Y).
\]
Thus, only a logarithmic pole along \(z=0\) occurs.

On every unramified component of \(\pi^{-1}(D)\), the map is the ordinary pullback
identification and has no pole. These local descriptions agree with the canonical
\etale\ identification on \(D\setminus\{x\}\). Hence, they glue uniquely, because the
target sheaf
\[
\Omega_X^1(\log\Delta)\otimes \pi_*(\mathcal{O}_Y)
\]
is torsion-free. This gives the global morphism
\[
\lambda_\pi:
\pi_*\Omega_Y^1
\longrightarrow
\Omega_X^1(\log\Delta)\otimes \pi_*(\mathcal{O}_Y).
\]

We also record that \(\lambda_\pi\) is \(\pi_*(\mathcal{O}_Y)\)-linear. This is immediate on the unramified
components. On the ramified component, let
\[
f=b_0(z)+t\,b_1(z)\in {\pi_*(\mathcal{O}_Y)}(D),
\qquad
\eta=(a_0(z)+t\,a_1(z))\,dt.
\]
Using the explicit formula above, one checks directly that
\[
\lambda_\pi(f\eta)=f\cdot \lambda_\pi(\eta),
\]
where \(\pi_*(\mathcal{O}_Y)\) acts on \(\Omega_X^1(\log\Delta)\otimes \pi_*(\mathcal{O}_Y)\) through the second factor.

It remains to prove the compatibility
\eqref{eq:lambda-dpi-compatibility}. Again, it suffices to check this locally. On the
ramified component, take a local section
\[
b_0(z)+t\,b_1(z)\in {\pi_*(\mathcal{O}_Y)}(D).
\]
Then
\[
d\pi_{\pi_*(\mathcal{O}_Y)}\bigl(dz\otimes (b_0(z)+t\,b_1(z))\bigr)
=
(b_0(z)+t\,b_1(z))\,d(t^2)
=
(2t\,b_0(z)+2z\,b_1(z))\,dt.
\]
Applying the formula for \(\lambda_\pi\), with
\[
a_0(z)=2z\,b_1(z),
\qquad
a_1(z)=2b_0(z),
\]
we get
\[
\lambda_\pi\bigl((2t\,b_0(z)+2z\,b_1(z))\,dt\bigr)
=
b_0(z)\,dz\otimes 1
+
b_1(z)\,dz\otimes t.
\]
Hence
\[
\lambda_\pi\circ d\pi_{\pi_*(\mathcal{O}_Y)}
\bigl(dz\otimes (b_0(z)+t\,b_1(z))\bigr)
=
dz\otimes (b_0(z)+t\,b_1(z)),
\]
which is precisely \((k\otimes\operatorname{id}_{\pi_*(\mathcal{O}_Y)})\) on this local section. The
verification on unramified components is immediate. Therefore
\[
\lambda_\pi\circ d\pi_{\pi_*(\mathcal{O}_Y)}
=
k\otimes \operatorname{id}_{\pi_*(\mathcal{O}_Y)}.
\]

Finally, composing \(\lambda_\pi\) with the spectral action gives
\[
\pi_*\Omega_Y^1
\xrightarrow{\lambda_\pi}
\Omega_X^1(\log\Delta)\otimes \pi_*(\mathcal{O}_Y)
\xrightarrow{\operatorname{id}\otimes m}
\Omega_X^1(\log\Delta)\otimes \End(E).
\]
The image lies in
\[
\Omega_X^1(\log\Delta)\otimes m(A)
\subset
\Omega_X^1(\log\Delta)\otimes \mathcal Z_\theta.
\]
Thus, this composition factors canonically through
\[
\mu_\pi:
\pi_*\Omega_Y^1
\longrightarrow
\Omega_X^1(\log\Delta)\otimes \mathcal Z_\theta .
\]
\end{proof}

Since \(\lambda_\pi\) is \(\pi_*(\mathcal{O}_Y)\)-linear, it can be tensored over \(\pi_*(\mathcal{O}_Y)\) with the
\(\pi_*(\mathcal{O}_Y)\)-module \(E=\pi_*L\). Using the canonical identification
\[
\pi_*(\Omega_Y^1\otimes L)
\cong
\pi_*\Omega_Y^1\otimes_A E,
\]
we obtain a morphism
\begin{equation}\label{eq:lambda-pi-L}
\lambda_{\pi,\,L}:
\pi_*(\Omega_Y^1\otimes L)
\longrightarrow
\Omega_X^1(\log\Delta)\otimes E.
\end{equation}
Explicitly, if
\[
\lambda_\pi(\eta)=\sum_r \alpha_r\otimes f_r,
\]
then
\[
\lambda_{\pi,\,L}(\eta\otimes s)
=
\sum_r \alpha_r\otimes f_r s.
\]

We now define the logarithmic spectral Atiyah class. Since \(L\) is a line bundle, its
Atiyah class is a scalar class
\[
\At(L)\in H^1(Y,\Omega_Y^1).
\]
Because \(\pi\) is finite, this may be viewed as a class
\[
\pi_*\At(L)\in H^1(X,\,\pi_*\Omega_Y^1).
\]

\begin{definition}\label{def:log-atiyah-class}
The \emph{logarithmic spectral Atiyah class} of \((E,\theta)\) is
\[
\At^{\log}_{\theta}(E)
:=
(\mu_\pi)_*\bigl(\pi_*\At(L)\bigr)
\in
H^1\bigl(X,\,\Omega_X^1(\log\Delta)\otimes \mathcal Z_\theta\bigr).
\]
\end{definition}

This class is canonical from the spectral data. It is obtained from the scalar Atiyah
class of the spectral line bundle by allowing logarithmic poles along the discriminant
and then applying the spectral action to the regularized centralizer.

\begin{lemma}\label{lem:scalar-extension-compatibility}
Under the natural identification
\[
\pi_*(\Omega_Y^1\otimes L)
\cong
\pi_*\Omega_Y^1\otimes_A E,
\]
the pushed-forward jet extension defining \(\beta_{\mathrm{spec}}(L,\,\pi)\) is obtained
from the scalar class
\[
\pi_*\At(L)\in H^1(X,\,\pi_*\Omega_Y^1)
\]
by the natural \(\pi_*(\mathcal{O}_Y)\)-module action on \(E\). Consequently,
\[
(\lambda_{\pi,\,L})_*\beta_{\mathrm{spec}}(L,\,\pi)
=
(\operatorname{id}\otimes q)_*(\mu_\pi)_*(\pi_*\At(L))
\]
in
\(
H^1\bigl(X,\,\Omega_X^1(\log\Delta)\otimes \End(E)\bigr),
\)
where
\(
q:\mathcal Z_\theta\hookrightarrow \End(E)
\)
is the inclusion.
\end{lemma}

\begin{proof}
Choose a sufficiently fine analytic open cover \(\{U_i\}\) of \(X\) such that \(L\) is
trivial on \(\pi^{-1}(U_i)\). Let
\[
g_{ij}\in A^\times(U_i\cap U_j)
\]
be the corresponding transition functions. Then the scalar class \(\pi_*\At(L)\) is
represented by the \v{C}ech cocycle
\[
d_Yg_{ij}\,g_{ij}^{-1}
\in
\pi_*\Omega_Y^1(U_i\cap U_j).
\]
The pushed-forward jet extension of \(L\) is represented by the same cocycle acting by
multiplication on the \(\pi_*(\mathcal{O}_Y)\)-module \(E\). Equivalently, under
\[
\pi_*(\Omega_Y^1\otimes L)
\cong
\pi_*\Omega_Y^1\otimes_A E,
\]
its transition cocycle is
\[
s\longmapsto
(d_Yg_{ij}\,g_{ij}^{-1})\otimes s.
\]

Applying \(\lambda_{\pi,\,L}\), this cocycle becomes the cocycle obtained by applying
\(\lambda_\pi\) to \(d_Yg_{ij}\,g_{ij}^{-1}\) and then letting the resulting element of
\(\pi_*(\mathcal{O}_Y)\) act on \(E\). After composing with
\[
q:\mathcal Z_\theta\hookrightarrow \End(E),
\]
this is exactly the cocycle obtained from
\[
\mu_\pi(d_Yg_{ij}\,g_{ij}^{-1})
=
(\operatorname{id}\otimes m)\lambda_\pi(d_Yg_{ij}\,g_{ij}^{-1}).
\]
Thus the two resulting cohomology classes agree in
\(
H^1\bigl(X,\,\Omega_X^1(\log\Delta)\otimes \End(E)\bigr).
\)
\end{proof}

\begin{proposition}\label{prop:global-log-compatibility}
Let
\(
k:\Omega_X^1\hookrightarrow \Omega_X^1(\log\Delta)
\)
be the natural inclusion, and let
\(
q:\mathcal Z_\theta\hookrightarrow \End(E)
\)
be the inclusion of the regularized centralizer sheaf. Then
\begin{equation}\label{eq:global-log-compatibility}
(k\otimes \operatorname{id}_{\End(E)})_*\At(E)
=
(\operatorname{id}\otimes q)_*\At^{\log}_{\theta}(E)
\end{equation}
in
\(
H^1\bigl(X,\,\Omega_X^1(\log\Delta)\otimes \End(E)\bigr).
\)
\end{proposition}

\begin{proof}
Recall the coefficient morphism from Section~\ref{sec:jets},
\[
\iota:
\Omega_X^1\otimes E
\longrightarrow
\pi_*(\Omega_Y^1\otimes L),
\]
which is induced by
\(
d\pi:\pi^*\Omega_X^1\longrightarrow \Omega_Y^1.
\)
By the compatibility \eqref{eq:lambda-dpi-compatibility}, after tensoring with \(L\) and
pushing forward, we obtain
\begin{equation}\label{eq:lambdatensor-iota}
\lambda_{\pi,\,L}\circ \iota
=
k\otimes \operatorname{id}_E:
\Omega_X^1\otimes E
\longrightarrow
\Omega_X^1(\log\Delta)\otimes E.
\end{equation}

By Corollary~\ref{cor:global-bridge}, we have
\(
\iota_*\At(E)=\beta_{\mathrm{spec}}(L,\,\pi)
\)
in
\(
\Ext_X^1\bigl(E,\pi_*(\Omega_Y^1\otimes L)\bigr).
\)
Applying the morphism induced by \(\lambda_{\pi,\,L}\), we get
\[
(\lambda_{\pi,\,L}\circ \iota)_*\At(E)
=
(\lambda_{\pi,\,L})_*\beta_{\mathrm{spec}}(L,\,\pi)
\]
in
\(
\Ext_X^1\bigl(E,\,\Omega_X^1(\log\Delta)\otimes E\bigr).
\)

Using \eqref{eq:lambdatensor-iota}, this becomes
\begin{equation}\label{eq:after-lambda}
(k\otimes \operatorname{id}_E)_*\At(E)
=
(\lambda_{\pi,\,L})_*\beta_{\mathrm{spec}}(L,\,\pi).
\end{equation}

Since \(E\) is locally free, we identify
\[
\Ext_X^1\bigl(E,\,\Omega_X^1(\log\Delta)\otimes E\bigr)
\cong
H^1\bigl(X,\,\Omega_X^1(\log\Delta)\otimes \End(E)\bigr).
\]

Under this identification, the left-hand side of \eqref{eq:after-lambda} is precisely
\[
(k\otimes \operatorname{id}_{\End(E)})_*\At(E).
\]
By Lemma~\ref{lem:scalar-extension-compatibility}, the right-hand side of
\eqref{eq:after-lambda} is
\[
(\operatorname{id}\otimes q)_*(\mu_\pi)_*(\pi_*\At(L)).
\]
By Definition~\ref{def:log-atiyah-class}, this is
\[
(\operatorname{id}\otimes q)_*\At^{\log}_{\theta}(E).
\]
Therefore
\[
(k\otimes \operatorname{id}_{\End(E)})_*\At(E)
=
(\operatorname{id}\otimes q)_*\At^{\log}_{\theta}(E)
\]
in
\(
H^1\bigl(X,\,\Omega_X^1(\log\Delta)\otimes \End(E)\bigr),
\)
as claimed.
\end{proof}

\begin{theorem}\label{thm:logarithmic-centralizer}
Let \((E,\theta)\) be a regular semisimple Higgs bundle on a smooth complex projective
curve \(X\), with smooth spectral curve
\(
\pi:Y\longrightarrow X
\)
and spectral line bundle \(L\). Assume that the
discriminant divisor \(\Delta\subset X\) is reduced. Then the logarithmic spectral Atiyah
class
\[
\At^{\log}_{\theta}(E)
\in
H^1\bigl(X,\,\Omega_X^1(\log\Delta)\otimes \mathcal Z_\theta\bigr)
\]
is a logarithmic refinement of the ordinary Atiyah class of \(E\) in the following global
sense:
\[
(k\otimes \operatorname{id}_{\End(E)})_*\At(E)
=
(\operatorname{id}\otimes q)_*\At^{\log}_{\theta}(E)
\]
in
\(
H^1\bigl(X,\,\Omega_X^1(\log\Delta)\otimes \End(E)\bigr),
\)
where
\[
k:\Omega_X^1\hookrightarrow \Omega_X^1(\log\Delta),
\qquad
q:\mathcal Z_\theta\hookrightarrow \End(E)
\]
are the natural inclusions. Consequently, after restriction to \(X^\circ=X\setminus\Delta\), one has
\[
(\operatorname{id}\otimes q^\circ)_*j^*\At^{\log}_{\theta}(E)
=
(\operatorname{id}\otimes m^\circ)_*
\beta^\circ_{\mathrm{spec}}(L,\,\pi)
=
j^*\At(E)
\]
in
\(
H^1\bigl(X^\circ,\Omega_{X^\circ}^1\otimes \End(E|_{X^\circ})\bigr),
\)
where
\[
q^\circ:\mathcal Z_\theta|_{X^\circ}\hookrightarrow \End(E|_{X^\circ})
\]
is the restricted inclusion.
\end{theorem}

\begin{proof}
The global equality is exactly Proposition~\ref{prop:global-log-compatibility}.  We prove the restriction statement. By Definition~\ref{def:log-atiyah-class},
\[
\At^{\log}_{\theta}(E)
=
(\mu_\pi)_*(\pi_*\At(L)).
\]
On \(X^\circ\), the map \(\lambda_\pi\) of Lemma~\ref{lem:log-coefficient-map} becomes
the \etale\ identification
\[
\pi^\circ_*\Omega_{Y^\circ}^1
\cong
\Omega_{X^\circ}^1\otimes \pi^\circ_*\cO_{Y^\circ}.
\]
Therefore, after applying the restricted inclusion
\[
q^\circ:\mathcal Z_\theta|_{X^\circ}\hookrightarrow \End(E|_{X^\circ}),
\]
the class \(j^*\At^{\log}_{\theta}(E)\) becomes the image of the scalar Atiyah class of
\(
L^\circ:=L|_{Y^\circ}
\)
under the spectral action
\[
\operatorname{id}\otimes m^\circ:
\Omega_{X^\circ}^1\otimes \pi^\circ_*\cO_{Y^\circ}
\longrightarrow
\Omega_{X^\circ}^1\otimes \End(E|_{X^\circ}).
\]
In the notation of Definition~\ref{def:beta-spec-etale}, this gives
\[
(\operatorname{id}\otimes q^\circ)_*j^*\At^{\log}_{\theta}(E)
=
(\operatorname{id}\otimes m^\circ)_*
\beta^\circ_{\mathrm{spec}}(L,\,\pi).
\]
By Theorem~\ref{thm:etale-invariant},
\(
(\operatorname{id}\otimes m^\circ)_*
\beta^\circ_{\mathrm{spec}}(L,\,\pi)
=
j^*\At(E).
\)
Hence
\[
(\operatorname{id}\otimes q^\circ)_*j^*\At^{\log}_{\theta}(E)
=
j^*\At(E)
\]
in
\(
H^1\bigl(X^\circ,\Omega_{X^\circ}^1\otimes \End(E|_{X^\circ})\bigr).
\)
This proves the theorem.
\end{proof}

\section*{Declarations}
\subsection{Funding} The authors received no financial support for the research, authorship, and/or publication of this paper.

\subsection{Conflicts of Interest} The authors have no competing interests to declare that are relevant to the content of this paper. No data are associated with this paper.

\subsection{Data Availability Statement} Data sharing not applicable to this article as no datasets were generated or analysed during the current study.
\bibliographystyle{alpha}
\bibliography{references}
\end{document}